\documentclass[12pt,leqno]{article}
\usepackage{amsmath, amsthm, amsfonts, amssymb, color}
\usepackage{mathrsfs}
\usepackage[notcite,notref]{showkeys}

\newtheorem{thm}{Theorem}[section]
\newtheorem{cor}[thm]{Corollary}

\newtheorem{prp}[thm]{Proposition}

\newtheorem{Remark}[thm]{Remark}
\theoremstyle{definition}

\newcommand{\scr}[1]{\mathscr #1}
\definecolor{wco}{rgb}{0.5,0.2,0.3}

\numberwithin{equation}{section}

\newcommand{\ua}{\uparrow}

\title{{\bf Spectral gaps for the O-U/Stochastic heat processes on path space over a Riemannian manifold with boundary}\footnote{Supported in
 part by  NNSFC (11371099).} }
\author{
{\bf     
Bo Wu }\\
\footnotesize {
 School  of Mathematical Sciences, Fudan
University, Shanghai 200433, China}
\\
\footnotesize{ 
wubo@fudan.edu.cn}}

\date{}

\begin{document}

\maketitle

\def\R{\mathbb R} \def\EE{\mathbb E} \def\Z{\mathbb Z} \def\ff{\frac} \def\ss{\sqrt}
\def\H{\mathbb H}
\def\dd{\delta} \def\DD{\Delta} \def\vv{\varepsilon} \def\rr{\rho}
\def\<{\langle} \def\>{\rangle} \def\GG{\Gamma} \def\gg{\gamma}
\def\ll{\lambda} \def\LL{\Lambda} \def\nn{\nabla} \def\pp{\partial}
\def\d{\text{\rm{d}}} \def\Id{\text{\rm{Id}}}\def\loc{\text{\rm{loc}}} \def\bb{\beta} \def\aa{\alpha} \def\D{\scr D}
\def\E{\scr E} \def\si{\sigma} \def\ess{\text{\rm{ess}}}
\def\beg{\begin} \def\beq{\beg}  \def\F{\scr F}
\def\Ric{\text{\rm{Ric}}}
\def\Var{\text{\rm{Var}}}
\def\Ent{\text{\rm{Ent}}}
\def\Hess{\text{\rm{Hess}}}\def\B{\scr B}
\def\e{\text{\rm{e}}} \def\ua{\underline a} \def\OO{\Omega} \def\b{\mathbf b}
\def\oo{\omega}     \def\tt{\tilde} \def\Ric{\text{\rm{Ric}}}
\def\cut{\text{\rm{cut}}} \def\P{\mathbb P} \def\ifn{I_n(f^{\bigotimes n})}
\def\fff{f(x_1)\dots f(x_n)} \def\ifm{I_m(g^{\bigotimes m})} \def\ee{\varepsilon}
\def\C{\scr C}
\def\M{\scr M}\def\ll{\lambda}
\def\X{\scr X}
\def\T{\scr T}
\def\A{\mathbf A}
\def\LL{\scr L}\def\LLL{\Lambda}
\def\gap{\mathbf{gap}}
\def\div{\text{\rm div}}
\def\dist{\text{\rm dist}}
\def\cut{\text{\rm cut}}
\def\supp{\text{\rm supp}}
\def\Cov{\text{\rm Cov}}
\def\Dom{\text{\rm Dom}}
\def\Cap{\text{\rm Cap}}\def\II{{\mathbb I}}\def\beq{\beg{equation}}
\def\sect{\text{\rm sect}}\def\H{\mathbb H}

\begin{abstract} Fang-Wu\cite{FW17} presented a explicit spectral gap for  the O-U process on path space over a Riemannian manifold without boundary under the bounded Ricci curvature conditions. In this paper, we will extend these results to the case of the Riemannian manifold with boundary. Moreover, we also derive the similar results for the stochastic heat process.
\end{abstract}

\vskip 2cm

\section{Introduction}\label{sect1}
Functional inequality is an important tool to study the spectral gaps for some diffusion operators in the analysis/stochastic analysis field, especially, for the case of infinite dimensional Riemannian path space. For the manifold without boundary,  
Fang\cite{F94} first established the Poincar\'e inequality for the O-U operator on Riemanian path space by the Clark-Ocean formula, after that the log-Sobolev inequality/(weak)Poincar\'e inequality have also been established for the O-U Dirichlet form, see e.g. \cite{AE95,CHL97,EL99,H97,H99, W04,CW14,N,WW18, W16, CT18} and references therein. Recently Naber\cite{N} gave some characterizations of the uniform bounds of Ricci curvature by the analysis of the path space. Motiviated by this work, Fang and Wu\cite{FW17} gave the explicit spectral gap of the O-U operator on path space under the Ricci curvature condition that
$K_2\leq\Ric^Z(:=\Ric +\nabla Z) \leq K_1$ and $K_2+K_1\geq 0$.
This condition $K_2+K_1\geq 0$ is removed by Cheng-Thalimaier\cite{CT18}.

For the manifold with boundary, Wang\cite{W14} proved the damped log-Sobolev inequality for the O-U process on path space, but some geometric informations are hidden in this inequality. 
In this article, our main aim is to present a estimate of the spectral gap for the O-U operator on path space over a manifold (possible with boundary) under  the curvature  and the second fundamental form conditions
\begin{equation}\label{eq1.1}K_2\leq\Ric^Z\leq K_1, \quad \sigma_2\leq \II\leq \sigma_1
\end{equation}
for some constants $K_2,K_1, \sigma_2,\sigma_1\in\R$.
In particular, our results cover Fang-Wu's results and Cheng-Thalmaier's results.  Moreover, we also obtain the estimate of the spectral gap for the stochastic heat process.

To state our main results, we need to introduce some notation. Let $M$ be a $d$-dimensional complete  Riemannian manifold possibly with a boundary $\partial M$ and $N$ be the inward unit normal vector field of $\pp M$. Let $L=\frac{1}{2}\DD+Z$ be the diffusion operator for some $C^1$ vector field $Z$, where $\Delta$ is the Laplace operator on $M$. 

Denote by the Riemannian path space:
$$W_x^T(M)=\{\gamma\in C([0,T]; M):\gamma_0=x\}.$$
Let $\rho$ be the Riemannian distance on M. Then $W_x^T(M)$ is a Polish space under the uniform distance
$$\rho_\infty(\gamma, \sigma) :=\sup_{t\in [0,T]} \rho(\gamma_t, \sigma_t), \quad \gamma,\sigma\in W_x^T(M).$$

Let $O(M)$ be the orthonormal frame bundle over $M$ and $\pi:O(M)\rightarrow M$ be the canonical projection.  
Furthermore, we choose a canonical orthonormal basis
$\{e_i\}_{1\le i\le d}$ on $\R^d$ and a standard orthonormal basis $H_i(u):=(H_{ue_i})_{1\le i\le d}$ for $u\in O(M)$ of horizontal vector fields on $O(M)$. Then the horizontal reflecting diffusion process is
the unique solution to the SDEs:
\begin{equation}\label{eq1.2}\beg{split} \d U_t^x=\, H_i(U_t^x)\circ\d W_t+H_{Z}(U_t^x)\d t+H_N(U_t^x)\d l_t^x,\ \ U_0^x\in O_x(M),\end{split}
\end{equation} 
where $W_t$ is the $d$-dimensional Brwonian motion on a complete filtration probability space $(\OO, \{\F_t\}_{t\ge 0}, \P)$, $H_{Z}$ and $H_{N}$ are the horizontal lift of $Z$ and $N$ respectively,  and $l^x_t$ is an adapted increasing process which increases only when $X_t^x:=\pi U_t^x\in \pp M$ which is  called the local time of $X_t^x$ on $\pp M$.  Then it is easy to know that $X_t^x$ solves the equation 
\begin{equation}\label{eq1.3}\beg{split} \d X_t^x= \, U_t^x\circ d W_t+ Z(X_t^x)\d t + N(X_t)\d l_t^x,\ \ X_0^x=x\end{split}
\end{equation}
up to the life time $\zeta$( the maximal time of the solution). 

Let $\F C^\infty_{T}$ be the space of bounded Lipschitz
continuous cylinder functions on $W_x^T(M)$, i.e. for every $F\in \F C^\infty_{T}$, there exist some $N\geq1$ and $0<t_1<t_2\cdots<t_N\leq T$,
$f \in C_{Lip}(M^N)$ such that $F(\gamma)=f(\gamma_{t_1},\cdots,\gamma_{t_N}), \gamma\in W_x^T(M)$ , where $C_{Lip}(M^N)$ is the collection of bounded Lipschitz continuous functions 
on $M^N$.
Suppose $\H$ is the standard Cameron-Martin space for $C([0, T]; \R^d)$, i.e.
 $$\H=
 \left\{h\in C([0,T];\mathbb{R}^d): h(0)=0,
\|h\|^2_{\mathbb{H_T}}:=\int_0^T|h_s'|^2\d s<\infty\right\}.$$

In order to construct O-U process on path space by the theory of Dirichlet form, we first introduce the damped Mallavin gradient given by Wang\cite{W11}. To  do that,  we will recommend a multiplicative functional $Q_{s,t}^x$, which is first introduced by Hsu \cite{H02} to investigate gradient estimate on $P_t$. 
For any fixed $s\ge 0$, $(Q_{s,t}^x)_{t\ge s}$ is an adapted right-continuous process on $\R^d\otimes \R^d$ such that $Q_{s,t}^x P_{U_t^x}=0$ if $X_t^x\in\pp M$ and
\beq\label{eq1.4} Q_{s,t}^x= \bigg(I-\int_s^t Q_{s,r}^x\big\{\Ric_Z(U_r^x)\d s + \II(U_r^x)\d l_r^x\big\}\bigg)
\Big(I-1_{\{X_t^x\in \pp M\}}P_{U_t^x}\Big),\end{equation} 
where $P_u:\R^d\to \R^d$ is the projection along $u^{-1}N$, i.e.
$$\<P_ua, b\>:=\<ua,N\>\<ub,N\>,\quad  a, b\in\R^d, u\in \cup_{x\in \pp M}O_x(M).$$ 
For every $F\in \scr F C_T^\infty$, by (4.2.1) in \cite{W11}, the damped gradient is defined by
\beq\label{eq1.5} \tt D_tF(X_{[0,T]}^x)= \sum_{i: t_i> t}Q_{t,t_i}^xU_{t_i}^{-1} \nn_if(X_{t_1}^x,\cdots, X_{t_N}^x),\ \
t\in [0,T].\end{equation}
Thus, the associated Mallavin gradient will defined as follows:
\beq\label{eq1.6}  D_tF(X_{[0,T]}^x)= \sum_{i: t_i> t} 
\Big(I-1_{\{X_{t_i}^x\in \pp M\}}P_{U_{t_i}^x}\Big)U_{t_i}^{-1} 
\nn_if(X_{t_1}^x,\cdots, X_{t_N}^x),\ \
t\in [0,T].\end{equation}

For any constants $K_2,K_1, \sigma_2,\sigma_1$ with $K_2\leq K_1, \sigma_2\leq \sigma_1$and each $t\in [0,T]$, let $\mu$ be the random measure on $[0,T]$ given by
$$\aligned\mu_t(\d r)&=\exp\left[-K_2(r-t)-\sigma_2(l_r^x-l_t^x)\right]\{(|K_1|\vee |K_2|)\d r+(|\sigma_1|\vee |\sigma_2|)\d l_r^x\} \}\\&=:\varphi_1(t,r,K_1,K_2,\sigma_1,\sigma_2)\d r+\varphi_1(t,r,K_1,K_2,\sigma_1,\sigma_2)\d l_r^x\endaligned$$
Denote by two measurable functions on $W_x^T(M)$
\beq\label{eq1.7}\aligned
&A_t=\Big(1+\mu([t,T])\Big) + \int_0^t
\Big(1+\mu([r,T])\Big)\varphi_1(r,t,K_1,K_2,\sigma_1,\sigma_2)\d r\\
&B_t= \int_0^t
\Big(1+\mu([r,T])\Big)\varphi_2(r,t,K_1,K_2,\sigma_1,\sigma_2)\d r.
\endaligned\end{equation}

Throughout the article, we assume that the (reflecting if $\pp M$ exists) $L$-diffusion process is non-explosive. Let $\P_x$ be the distribution of the $L$-diffusion process $X^x$ starting from a fixed point $x$ up to some fixed time $T > 0.$ Then $\P_x$ is a probability measure on the Riemannian path space $W_x^T(M).$
Define the following quadratic form by
\beq\label{eq1.8}\aligned \E_{\sigma_1,\sigma_2}^{K_1,K_2}(F,F):=\int_{W_x^T(M)}\int^T_0
|D_tF|^2(A_t\d t+B_t\d l_t)\d \P_x.
\endaligned\end{equation}
The following Logarithmic Sobolev inequality is the main result of this paper.

\beg{thm}\label{T1.1} Assume that $K_2\leq \Ric^Z \leq K_1$ and $\sigma_2\leq \II \leq \sigma_1$. Then the following Logarithmic Sobolev inequality holds  
\begin{equation}\label{eq1.9}
\mathbb{E}\bigg(F^2\log \frac{F^2}{\|F\|^2_{L^2}}\bigg)\leq\E_{\sigma_1,\sigma_2}^{K_1,K_2}(F,F),\quad F\in \F C^\infty_{T}.\end{equation}
\end{thm}

By the above Theorem \ref{T1.1}, we obtain the following Corollary for two special cases.

\begin{cor}\label{C1.2}$(a)$ Assume that $M$ is a Riemannian manifold without a boundary and $K_1\leq \Ric^Z\leq K_2$, then the Logarithmic Sobolev inequality holds
\begin{equation}\label{eq1.10}
\mathbb{E}\bigg(F^2\log \frac{F^2}{\|F\|^2_{L^2}}\bigg)\leq C(T,K_1,K_2)\int_{W_x^T(M)}\int^T_0|D_tF|^2\d t\d \P_x,\quad F\in \F C^\infty_b(M).\end{equation}
In particular, when $K_2<0$, we have
\begin{equation}\label{eq1.11}
Spect(\L)^{-1} \leq \frac{1}{2}+\frac{1}{2}\left(1+\frac{|K_1|\vee|K_2|}{K_2}\Big[1-\e^{-K_2T}\Big]\right)^2.
\end{equation} When $K_2>0$, we have
\begin{equation}\label{eq1.12}\aligned
SG(\L)^{-1}&\leq(1+\beta)^2-2\sqrt{\Big(\beta+\frac{\beta^2}{2}\Big)\Big(\beta+\beta^2-\frac{\beta^2}{2}\e^{-K_2T}\Big)}\ \e^{-\frac{K_2T}{2}},
\endaligned\end{equation}
where $\beta=\frac{|K_1|\vee|K_2|}{K_2}$.

$(b)$ Let $M$ be Ricci flat Riemannian manifold with boundary, and we assume that the second fundamental form satisfies $\sigma_2\leq \II\leq \sigma_1$,  then 

When $\sigma_2\geq0$, we get that for any $F\in \F C^\infty_{T}$
  \begin{equation}\label{eq1.13}
\aligned
\mathbb{E}\bigg(F^2\log \frac{F^2}{\|F\|^2_{L^2}}\bigg)\leq&\int_{W_x^T(M)}(1+\sigma_1l_T^x)\int^T_0|D_tF|^2\d t\d \P_x\\&+\int_{W_x^T(M)}\left(\sigma_1(1+\sigma_1l_T^x)T\right)\int^T_0|D_tF|^2\d l_t\d \P_x.\endaligned\end{equation}
When $\sigma_2<0$, we get that for any $F\in \F C^\infty_{T}$
  \begin{equation}\label{eq1.14}
\aligned
\mathbb{E}\bigg(F^2\log \frac{F^2}{\|F\|^2_{L^2}}\bigg)\leq&\int_{W_x^T(M)}(1+(|\sigma_1|\vee |\sigma_2|)\exp\left[(-\sigma_2+\varepsilon)l_T^x\right])\int^T_0|D_tF|^2\d t\d \P_x\\&+\int_{W_x^T(M)}\left(2(|\sigma_1|\vee |\sigma_2|)^2\exp\left[(-\sigma_2+\varepsilon)l_T^x\right]\right)\int^T_0|D_tF|^2\d l_t\d \P_x\endaligned\end{equation} 
for some constant $\varepsilon>0$.
\end{cor}

\begin{Remark}
$(1)$ Wang \cite{W11} proved that the damped Logarithmic Sobolev inequality.

$(2)$ When $M$ is a Riemannian manifold without boundary, and $K_2\leq\Ric^Z \leq K_1, K_2+K_1\geq$, Fang-Wu \cite{FW17} first proved \eqref{eq1.10}, later , this result had been extended to the general case of $K_2$ and $K_1$ by Cheng-Thalimaier\cite{CT18}.
\end{Remark}

The rest of this paper is organized as follows: In Section 2, we will prove Theorem \ref{T1.1} and Corollary \ref{C1.2}. The estimate of the spectral gap for the stochastic heat process will be presented in Section 3.

\section{Proofs of Theorem \ref{T1.1} and Corollary \ref{C1.2}}
\subsection{Proof of Theorem \ref{T1.1}}
\beg{proof}[Proof of Theorem \ref{T1.1}]
By Theorem 4.4 in \cite{W11}, we know that the following damped logarithmic Sobolev inequality holds 
\begin{equation}\label{eq2.1}
\mathbb{E}\bigg(F^2\log \frac{F^2}{\|F\|^2_{L^2}}\bigg)\leq 2\int_{W^T_x(M)}\int^T_0|\tilde{D}_tF|^2\d t\d \P_x.
\end{equation}
Therefore it suffices to show that
\begin{equation}\label{eq2.2}
\int^T_0|\tilde{D}_tF|^2\d t\leq\int^T_0
|D_tF|^2(A_t\d t+B_t\d l_t).
\end{equation}
By using the assumptions of $K_2\leq \Ric^Z \leq K_1$ and $\sigma_2\leq \II \leq \sigma_1$, we have
$$\Vert\Ric^Z\Vert\leq |K_1|\vee |K_2|,\quad \Vert \II\Vert\leq |\sigma_1|\vee |\sigma_2|,$$ 
Combining this with \eqref{eq1.14} 
$$Q_{s,t}^x= \bigg(I-\int_s^t Q_{s,r}^x\big\{\Ric^Z(U_r^x)\d s + \II(U_r^x)\d l_r^x\big\}\bigg)
\Big(I-1_{\{X_t^x\in \pp M\}}P_{U_t^x}\Big).$$
and \cite[Theorem 3.2.1]{W14}, it is easy to derive that
\beq\label{eq2.3} \left\|Q_{t,r}^x\right\|\leq \exp\left[-K_2(r-t)-\sigma_2(l_r^x-l_t^x)\right].
\end{equation}   
By the definition of the damped gradient, we get
\beq\label{eq2.4}\aligned
&\tt D_tF(X_{[0,T]}^x)= \sum_{i: t_i> t}Q_{t,t_i}^xU_{t_i}^{-1} \nn_if(X_{t_1}^x,\cdots, X_{t_N}^x)=D_tF(X_{[0,T]}^x)-\\
&\sum_{i: t_i> t}\int_t^{t_i} Q_{t,r}^x\big\{\Ric^Z(U_r^x)\d s + \II(U_r^x)\d l_r^x\big\}
\Big(I-1_{\{X_t^x\in \pp M\}}P_{U_{t_i}^x}\Big)U_{t_i}^{-1} \nn_if(X_{t_1}^x,\cdots, X_{t_N}^x)\\
&=D_tF(X_{[0,T]}^x)-\int_t^T Q_{t,r}^x\big\{\Ric^Z(U_r^x)D_rF(X_{[0,T]}^x)\d r + \II(U_r^x)D_rF(X_{[0,T]}^x)\d l_r^x\big\}.
\endaligned\end{equation}
Then we have
\beq\label{eq2.5}\aligned
&|\tt D_tF|(X_{[0,T]}^x)\leq|D_tF|(X_{[0,T]}^x)\\
&+\int_t^T\exp\left[-K_2(r-t)-\sigma_2(l_r^x-l_t^x)\right]\{(|K_1|\vee |K_2|)\d r+(|\sigma_1|\vee |\sigma_2|)\d l_r^x\} |D_rF|(X_{[0,T]}^x)\}\\&=|D_tF|(X_{[0,T]}^x)+\int_t^T|D_rF|\mu_t(\d r).
\endaligned\end{equation}
The H\"{o}lder's inequality implies that
\beq\label{eq2.6}\aligned
|\tt D_tF|^2(X_{[0,T]}^x)&\leq\Big(1+\mu_t([t,T])\Big)\bigg(|D_tF|^2+\int_t^T
|D_rF|^2\mu_t(\d r)\bigg)
\endaligned\end{equation}
Thus, we obtain
\beq\label{eq2.7}\aligned
&\int^T_0|\tt D_tF|^2\d t\leq\int^T_0\Big(1+\mu([t,T])\Big)|D_tF|^2\d t +\int^T_0 \int_t^T
\Big(1+\mu([t,T])\Big)|D_rF|^2\mu_t(\d r)\d t\\
&=\int^T_0\Big(1+\mu([t,T])\Big)|D_tF|^2\d t +\int^T_0 \int_t^T
\Big(1+\mu([t,T])\Big)|D_rF|^2\varphi_1(t,r,K_1,K_2,\sigma_1,\sigma_2)\d r\d t\\
&+\int^T_0 \int_t^T
\Big(1+\mu([t,T])\Big)|D_rF|^2\varphi_2(t,r,K_1,K_2,\sigma_1,\sigma_2)\d l_r\d t\\
&=\int^T_0\Big(1+\mu([t,T])\Big)|D_tF|^2\d t +\int^T_0 \int_0^r
\Big(1+\mu([t,T])\Big)|D_rF|^2\varphi_1(t,r,K_1,K_2,\sigma_1,\sigma_2)\d r\d t\\
&+\int^T_0 \int_0^r
\Big(1+\mu([t,T])\Big)|D_rF|^2\varphi_2(t,r,K_1,K_2,\sigma_1,\sigma_2)\d l_r\d t\\
&=\int^T_0\Big(1+\mu([t,T])\Big)|D_tF|^2\d t +\int^T_0 \int_0^t
\Big(1+\mu([r,T])\Big)\varphi_1(r,t,K_1,K_2,\sigma_1,\sigma_2)\d r|D_tF|^2\d t\\
&+\int^T_0 \int_0^t
\Big(1+\mu([r,T])\Big)\varphi_2(r,t,K_1,K_2,\sigma_1,\sigma_2)\d r|D_tF|^2\d l_t\\
&=\int^T_0
|D_tF|^2(A_t\d t+B_t\d l_t).
\endaligned\end{equation}
Up to now, we complete the proof.
\end{proof}

\subsection{Proof of Corollary \ref{C1.2}}
To prove Corollary \ref{C1.2}, we need some preparations. 
Let $\beta=\frac{|K_1|\vee |K_2|}{K_2}$ and define 
\begin{equation}\label{eq2.8}
\aligned
&\Lambda(t,T)=1+\beta\Big[1-\exp\left[-K_2(T-t)\right]\Big]+\left(\beta+\beta^2\right)\Big[1-\exp\left(-K_2t\right)\Big]\\&
~~~~~~~+\frac{\beta^2}{2}\left[\exp\left(-K_2(t+T)\right)-
\exp\left(-K_2(T-t)\right)\right]\endaligned\end{equation}
and 
$$C(T,K_1,K_2):=\sup_{t\in[0,T]}\Lambda(t,T).$$
Similar to the proof of Proposition 3.3 in Fang-Wu\cite{FW17} for the case of $K_1+K_2\geq0$, in the following we will discuss monotonicity of the function $\Lambda(\cdot,T)$.

\begin{prp}\label{p2.1}  $(1)$ If $K_2<0$, then $t\rightarrow\Lambda(t,T)$ is strictly increasing over $[0,T]$.
$(2)$ If $K_2\geq0$, then the maximum is attained at a point $t_0$ in $(0, T)$.
\end{prp}

\begin{proof}According to the definition \eqref{eq2.8} of $\Lambda(t,T)$, we get
$$ \Lambda(0,T)=1+\beta\Big(1-\e^{-K_2T}\Big)$$
and
$$\aligned\Lambda(T,T)&=1+\left(\beta+\beta^2\right)\Big[1-\exp\left(-K_2T\right)\Big]+\frac{\beta^2}{2}\left[\e^{-2K_2T}-
1\right]\\
&=\frac{1}{2}+\frac{1}{2}\bigg[1+\beta\Big(1-\e^{-K_2T}\Big)\bigg]^2=\frac{1}{2}+\frac{1}{2}\Lambda^2(0,T).
 \endaligned$$
 In particular, the second in the above implies that $\Lambda(T,T)\geq\Lambda(0,T)$.

Next, we take the derivative of $\Lambda(t,T)$  with respect to $t$,
$$\aligned \Lambda'(t,T)
&=-\beta K_2\e^{-K_2(T-t)}+\left(\beta+\beta^2\right)K_2\e^{-K_2t}
-\frac{\beta^2}{2}K_2\left[\e^{-K_2(t+T)}+
\e^{-K_2(T-t)}\right].\endaligned$$
Then we have
\begin{equation}\label{eq2.9}\aligned \Lambda'(0,T)&=-\beta K_2\e^{-K_2T}+\left(\beta+\beta^2\right)K_2
-\beta^2K_2\e^{-K_2T}
\\&
=\beta K_2(1+\beta)(1-\e^{-K_2T})\geq0;
 \endaligned\end{equation}
and
\begin{equation}\label{eq2.10}\aligned \Lambda'(T,T)&=
-\beta K_2+\left(\beta+\beta^2\right)K_2\e^{-K_2T}
-\frac{\beta^2}{2}K_2\left[\e^{-2K_2T}+
1\right]\\&=-\beta K_2\left(1-\left(1+\beta\right)\e^{-K_2T}
+\frac{\beta}{2}\left[\e^{-2K_2T}+
1\right]\right)\\&=-\beta K_2[1-\e^{-K_2T}]-\frac{\beta^2K_2}{2}[1-\e^{-K_2T}]^2.
 \endaligned\end{equation}
Noting that

\begin{equation}\label{eq2.11}
 \left\{ \begin{array}{ll}
\Lambda'(T,T)>0  \qquad   &if ~K_2<0,\\
\Lambda'(T,T)<0 & if ~K_2>0.
\end{array}\right.
\end{equation}

Now we look for $t\in [0,T]$ such that $\Lambda'(t,T)=0$. We have

\begin{equation}\label{eq2.12}\aligned &\Lambda'(t,T)=0 \\
&\Leftrightarrow~~-\e^{-K_2(T-t)}+\left(1+\beta\right)\e^{-K_2t}
-\frac{\beta}{2}\left[\e^{-K_2(t+T)}+
\e^{-K_2(T-t)}\right]=0\\
&\Leftrightarrow~~-\e^{-K_2T}\e^{2K_2t}+\left(1+\beta\right)
-\frac{\beta}{2}\left[\e^{-K_2T}+
\e^{-K_2T}\e^{2K_2t}\right]=0\\&\Leftrightarrow~~\Big(1+\frac{\beta}{2}\Big)\e^{-K_2T}\e^{2K_2t}
=1+\beta-\frac{\beta}{2}\e^{-K_2T}.\endaligned\end{equation}
Therefore there exists at most one $t$ such that  $\Lambda'(t,T)=0$.
For the case where $K_2<0$, if there exists $t_0\in(0,T)$ such that $\Lambda(t_0,T)<0$.
Then by (\ref{eq2.9}) and \eqref{eq2.11}, the equation $\Lambda'(t,T)=0$ has at least two solutions, it is impossible.
Therefore for $K_2<0$, $\Lambda'(t,T)\geq0$. For $K_2>0$, we suppose $t_0$ such that $\Lambda'(t_0,T)=0$, then by (\ref{eq2.12})
$$\e^{2K_2t_0}=\Big(1+\frac{\beta}{2+\beta}\Big(1-\e^{-K_2T}\Big)\Big)\e^{K_2T}.$$
Thus the proof is completed.\end{proof}

By the above Proposition \ref{p2.1}, it is easy to obtain the following Proposition \ref{p2.2}.

\begin{prp}\label{p2.2} $(i)$ If $K_2>0$,
\begin{equation}\label{eq2.13}
\begin{split}
 \sup_{t\in [0,T]}\Lambda(t,T)
&=(1+\beta)^2-\Big(\beta+\frac{\beta^2}{2}\Big)\sqrt{1+\frac{\beta}{2+\beta}\Big(1-\e^{-K_2T}\Big)}\ \e^{-\frac{K_2T}{2}}\\
&-\frac{\Big(\beta+\beta^2-\frac{\beta^2}{2}\e^{-K_2T}\Big)}{\sqrt{1+\frac{\beta}{2+\beta}\Big(1-\e^{-K_2T}\Big)}}
\e^{-\frac{K_2T}{2}}.
\end{split}
\end{equation}
$(ii)$ If $K_2<0$,
\begin{equation}\label{eq2.14}
\sup_{t\in[0,T]}\Lambda(t,T)=\frac{1}{2}+\frac{1}{2}\left(1+\frac{|K_1|\vee |K_2|}{K_2}\Big[1-\e^{-K_2T}\Big]\right)^2.
\end{equation}

\end{prp}

\beg{proof}[Proof of Corollary \ref{C1.2}]
$(a)$ Since $M$ is a Riemannian manifold without boundary, then the local time $l_t=0$, thus 
we have

$$\aligned\mu_t(\d r):&=(|K_1|\vee |K_2|)\exp\left[-K_2(r-t)\right]\d r.\endaligned$$
Then
\beq\label{eq2.15}\aligned
 \mu([t,T])=\frac{(|K_1|\vee |K_2|)}{K_2}\left[1-\exp\left[-K_2(T-t)\right]\right].\endaligned\end{equation}
Which implies that
\beq\label{eq2.16}\aligned
 &\varphi_1(r,t,K_1,K_2,\sigma_1,\sigma_2)=(|K_1|\vee |K_2|)\exp\left[-K_2(t-r)\right]\\
 &\varphi_2(r,t,K_1,K_2,\sigma_1,\sigma_2)=0.\endaligned\end{equation}
Then by \eqref{eq2.15} and the first equality of \eqref{eq2.16},
\beq\label{eq2.17}\aligned
&\int_0^t
\Big(1+\mu([r,T])\Big)\varphi_1(r,t,K_1,K_2,\sigma_1,\sigma_2)\d r\\
&=(|K_1|\vee |K_2|)\int_0^t
\left(1+\frac{(|K_1|\vee |K_2|)}{K_2}\left[1-\exp\left[-K_2(T-r)\right]\right]\right)\exp\left[-K_2(t-r)\right]\d r\\&
=\left(\frac{(|K_1|\vee |K_2|)}{K_2}+\frac{(|K_1|^2\vee |K_2|^2)}{K_2^2}\right)\Big[1-\exp\left(-K_2t\right)\Big]\\&
~~~~~~~-\frac{(|K_1|^2\vee |K_2|^2)}{2K_2^2}\exp\left(-K_2(t+T)\right)
\left(\exp\left(2K_2t\right)-1\right).
\endaligned\end{equation}
Thus we get
\beq\label{eq2.18}\aligned
&B_t= \int_0^t
\Big(1+\mu([r,T])\Big)\varphi_2(r,t,K_1,K_2,\sigma_1,\sigma_2)\d r=0\\
&A_t=\Big(1+\mu([t,T])\Big) + \int_0^t
\Big(1+\mu([r,T])\Big)\varphi_1(r,t,K_1,K_2,\sigma_1,\sigma_2)\d r\\
&=1+\frac{(|K_1|\vee |K_2|)}{K_2}\left[1-\exp\left[-K_2(T-t)\right]\right]
\\&~~~~~~~+\left(\frac{(|K_1|\vee |K_2|)}{K_2}+\frac{(|K_1|^2\vee |K_2|^2)}{K_2^2}\right)\Big[1-\exp\left(-K_2t\right)\Big]\\&
~~~~~~~+\frac{(|K_1|^2\vee |K_2|^2)}{2K_2^2}\left[\exp\left(-K_2(t+T)\right)-
\exp\left(-K_2(T-t)\right)\right].
\endaligned\end{equation}
From which we have
\begin{equation}\label{eq2.19}
\int^T_0
|D_tF|^2(A_t\d t+B_t\d l_t)\leq \int^T_0\Lambda(t,T)
|D_tF|^2\d t.
\end{equation}
Then \eqref{eq1.9}, \eqref{eq1.10} and \eqref{eq1.11} come from Theorem \ref{T1.1} and Proposition \ref{p2.2}.

$(b)$ By the assumption of $M$, we know that
\beq\label{eq2.20}\aligned
 \mu_t(\d r)&=(|\sigma_1|\vee |\sigma_2|)\exp\left[-\sigma_2(l_r^x-l_t^x)\right]\d l_r^x.\endaligned\end{equation}
 Thus,
 \beq\label{eq2.21}\aligned
 \mu([t,T])=(|\sigma_1|\vee |\sigma_2|)\exp\left[\sigma_2l_t^x\right]\int^T_t\exp\left[-\sigma_2 l_r^x\right]\d l_r^x.\endaligned\end{equation}
In  addition, \eqref{eq2.20} implies that 
 \beq\label{eq2.22}\aligned
 &\varphi_1(r,t,K_1,K_2,\sigma_1,\sigma_2)=0\\
 &\varphi_2(r,t,K_1,K_2,\sigma_1,\sigma_2)=(|\sigma_1|\vee |\sigma_2|)\exp\left[-\sigma_2(l_t^x-l_r^x)\right].\endaligned\end{equation}
 Then, by the definition of $A_t$ and $B_t$,
  \beq\label{eq2.23}\aligned
 &A_t=1+(|\sigma_1|\vee |\sigma_2|)\exp\left[\sigma_2l_t^x\right]\int^T_t\exp\left[-\sigma_2 l_r^x\right]\d l_r^x\\
 &B_t=(|\sigma_1|\vee |\sigma_2|)\int^t_0\left(1+(|\sigma_1|\vee |\sigma_2|)\exp\left[\sigma_2l_r^x\right]\int^T_r\exp\left[-\sigma_2 l_u^x\right]\d l_u^x\right)
 \\&~~~~~~~~~~~\times \exp\left[-\sigma_2(l_t^x-l_r^x)\right]\d r.\endaligned\end{equation}
Since $l_t^x$ is a increasing process, we have 
 \begin{equation}\label{eq2.24}
A_t\leq\begin{cases}
1+\sigma_1l_T^x,~~~~~\quad\quad\quad\quad\quad\quad\quad\quad\quad\quad\quad\quad\text{if}\ \sigma_2\geq0,\\
1+(|\sigma_1|\vee |\sigma_2|)\exp\left[(-\sigma_2+\varepsilon)l_T^x\right],\quad\quad\quad\text{if}\ \sigma_2<0
\end{cases}
\end{equation}
and 
 \begin{equation}\label{eq2.25}
B_t\leq\begin{cases}
\sigma_1(1+\sigma_1l_T^x)T,~~~~\quad\quad\quad\quad\quad\quad\quad\quad\quad\text{if}\ \sigma_2\geq0,\\
2(|\sigma_1|\vee |\sigma_2|)^2\exp\left[(-\sigma_2+\varepsilon)l_T^x\right],\quad\quad\quad\text{if}\ \sigma_2<0
\end{cases}
\end{equation}
By Theorem \ref{T1.1}, when $\sigma_2\geq0$, we get
  \begin{equation}\label{eq2.26}
\aligned&
\mathbb{E}\bigg(F^2\log \frac{F^2}{\|F\|^2_{L^2}}\bigg)\\&\leq\int_{W_x^T(M)}(1+\sigma_1l_T^x)\int^T_0|D_tF|^2\d t\d \P_x\\&+\int_{W_x^T(M)}\left(\sigma_1(1+\sigma_1l_T^x)T\right)\int^T_0|D_tF|^2\d l_t\d \P_x,\endaligned\end{equation}
and when $\sigma_2<0$, we get
  \begin{equation}\label{eq2.27}
\aligned&
\mathbb{E}\bigg(F^2\log \frac{F^2}{\|F\|^2_{L^2}}\bigg)\\&\leq\int_{W_x^T(M)}(1+(|\sigma_1|\vee |\sigma_2|)\exp\left[(-\sigma_2+\varepsilon)l_T^x\right])\int^T_0|D_tF|^2\d t\d \P_x\\&+2\int_{W_x^T(M)}\left(|\sigma_1|\vee |\sigma_2|)^2\exp\left[(-\sigma_2+\varepsilon)l_T^x\right]\right)\int^T_0|D_tF|^2\d l_t\d \P_x.\endaligned\end{equation}
\end{proof}

\beg{cor} Let $M:=\left\{x=(x_1,\cdots,x_d)\in \R^d: a_1x_1+\cdots+a_dx_d\geq c\right\}$ for some constant $c\in \R$, then $M$ is  a Riemannian manifold without boundary and $\Ric=0$ with $ \II=0$, thus
we have
\begin{equation}\label{eq2.28}
\aligned
\mathbb{E}\bigg(F^2\log \frac{F^2}{\|F\|^2_{L^2}}\bigg)\leq&\int_{W_x^T(M)}\int^T_0|D_tF|^2\d t\d \P_x.\endaligned\end{equation}
\end{cor}

\section{Stochastic heat equation}
In this section, we will consider the spectral gap for the stochastic heat equation on a Riemannian manifold with boundary. Before moving on, let's introduce some notation.

The stochastic heat equation on Riemannian manifold had been studied detailed by \cite{RWZZ17}(see also \cite{H16}). Here they introduced some notation. In particular, the classical cylinder function depending on finite times is not in the domain of generator associated to the stochastic heat equation. Thus, we need to introduce a class of new cylinder function $\F C_b^1$ on $W_x^T(M),$ i.e. for every $F\in \F C_b^1$, there exist some $m\geq1, ~m\in \mathbb{N},~ f\in C_b^1(\mathbb{R}^m), g_i\in C_b^{0,1}([0,1]\times M)$, $i=1,...,m$,
such that
\begin{equation}\label{eq3.1}\aligned
F(\gamma)=f\left(\int_0^1  g_1(s,\gamma_s) \d s,\int_0^1  g_2(s,\gamma_s) \d s,...,\int_0^1  g_m(s,\gamma_s) \d s\right),\quad \gamma\in W_x^T(M),\endaligned\end{equation}
where $C_b^{0,1}([0,1]\times M)$ denotes the functions which are continuous w.r.t. the first variable and differentiable w.r.t. the second variable with continuous derivatives.

For any $F\in \F C_b^1$ with \eqref{eq2.2} form and $h\in L^2([0,1];\mathbb{R}^d)$, according to Wang\cite{W11}, the damped Malliavin gradient  of $F$ is given by
$$\dot{\tt D}F(s)(\gamma):=\sum_{j=1}^m\hat{\partial}_jf(\gamma)\int_s^TQ_{s,u}U_u^{-1}(\gamma)\nabla g_j(u,\gamma_u)\d u,\quad \gamma\in W_x^T(M).$$

Let $\tt \nabla F$ be the damped $L^2$-gradient of $F$, and since 
$$\aligned &\int_0^T\left\langle\dot{\tt D}F(s),h'_s\right\rangle \d s=\langle h,\mathbf{D}F\rangle_{\H}=D_hF=\langle h,\mathbf{D}F\rangle\>_{L^2}\\
&=\int_0^T\left\langle \tt \nabla F(s),h_s\right\rangle \d s=\int_0^T\left\langle \tt \nabla F(s),\int^s_0h'_u\d u\right\rangle\d s\\
&=\int_0^T\int^s_0\left\langle \tt \nabla F(s),h'_u\right\rangle\d u\d s 
=\int_0^T\int^T_u\left\langle \tt \nabla F(s),h'_u\right\rangle\d s\d u \\
&= \int_0^T\left\langle\int^T_s \tt \nabla F(u)\d u,h'_s\right\rangle\d s.\endaligned$$
Then, we have
$$\int^T_s \tt \nabla F(u)\d u=\dot{\tt D}F(s).$$
Thus,
$$\tt \nabla F(s)=\sum_{j=1}^m\hat{\partial}_jf(\gamma)Q_{s,T}U_s^{-1}(\gamma)\nabla g_j(s,\gamma_s).$$
The $L^2$-gradient of $F$ is defined by
$$\nabla F(s)=\sum_{j=1}^m\hat{\partial}_jf(\gamma)U_s^{-1}(\gamma)\nabla g_j(s,\gamma_s).$$
By Lemma 4.3.2 in \cite{W11}, we have
\begin{equation}\label{eq3.2}\aligned 
F&=F+\sqrt{2}\int^T_0\langle\dot{\tt D}F(s), \d B_s\rangle\\
&=F+\sqrt{2}\int^T_0\left\langle\int^T_s \tt \nabla F(u)\d u, \d B_s\right\rangle\\
&=F+\sqrt{2}\int^T_0\left\langle\int^T_sQ_{s,u} \nabla F(u)\d u, \d B_s\right\rangle
\endaligned\end{equation}
By the standard the procedure, we have
\begin{equation}\label{eq3.3}
\mathbb{E}\bigg(F^2\log \frac{F^2}{\|F\|^2_{L^2}}\bigg)\leq 2\int_{W^T_x(M)}\int^T_0\left|\int^T_s Q_{s,u} \nabla F(u)\d u\right|^2\d s\d \P_x.
\end{equation}
By \eqref{eq2.3} and H\"{o}lder's inequality, we get
\beq\label{eq3.4}\aligned\left|\int^T_t Q_{s,u} \nabla F(u)\d u\right|^2&\leq\left|\int^T_s \e^{-K(u-s)-\sigma(l_u^x-l_s^x)} \nabla F(u)\d u\right|^2\\&
\leq\int^T_s \e^{-K(u-s)-\sigma(l_u^x-l_s^x)}\d u\int^T_s \e^{-K(u-s)-\sigma(l_u^x-l_s^x)} |\nabla F|^2(u)\d u
\\&
=\int^T_s \e^{-K(u-s)-\sigma(l_u^x-l_s^x)} \d u\int^T_s \e^{-K(u-s)-\sigma(l_u^x-l_s^x)} |\nabla F|^2(u)\d u.
\endaligned\end{equation}  
Let 
$$\varphi(s)=\int^T_s \e^{-K(u-s)-\sigma(l_u^x-l_s^x)} \d u.$$
Then by changing the order of integration we obtain
$$\aligned 
&\int_{W^T_x(M)}\int^T_0\left|\int^T_s Q_{s,u} \nabla F(u)\d u\right|^2\d s\d \P_x\leq \int_{W^T_x(M)}\int^T_0\varphi(s)\int^T_s \e^{-K(u-s)-\sigma(l_u^x-l_s^x)} |\nabla F|^2(u)\d u\d s\d \P_x
\\
&=\int_{W^T_x(M)}\int^T_0A(s) |\nabla F|^2(s)\d s\d \P_x,\endaligned$$
where
$$\aligned A(s)=\int^s_0\varphi(u) \e^{-K(s-u)-\sigma(l_s^x-l_u^x)}\d u.\endaligned$$
Thus, we get the following Logarithmic Sobolev inequality.

\beg{thm}\label{T3.1} Assume that $\Ric^Z\geq K$ and $\II\geq \sigma$. Then the following Logarithmic Sobolev inequality holds  
\begin{equation}\label{eq3.5}
\mathbb{E}\bigg(F^2\log \frac{F^2}{\|F\|^2_{L^2}}\bigg)\leq2\int_{W^T_x(M)}\int^T_0A(s) |\nabla F|^2(s)\d s\d \P_x,\quad F\in \F C^1_b.\end{equation}
\end{thm}

\begin{cor}\label{c3.2}$(a)$ Assume that $M$ is a Riemannian manifold with a convex boundary and $K\leq \Ric^Z$, then the Logarithmic Sobolev inequality holds
\begin{equation}\label{eq3.6}
\mathbb{E}\bigg(F^2\log \frac{F^2}{\|F\|^2_{L^2}}\bigg)\leq C(T,K)\int_{W_x^T(M)}\int^T_0|\nabla F|^2\d t\d \P_x,\quad F\in \F C^\infty_b(M),\end{equation}
where
 \begin{equation*}\label{e0a}
C(T,K)=\begin{cases}
\frac{1}{2K^2}\bigg[2-\frac{2-e^{-KT}}{\sqrt{2\e^{KT}-1}}-e^{-KT}\sqrt{2\e^{KT}-1}\bigg],~~\quad\quad\quad\text{if}\ K\geq0,\\
\frac{1}{K^2}\big[1-e^{-KT}\big]^2,~~~~\quad\quad\quad\quad\quad\quad\quad\quad\quad\quad\quad\quad\text{if}\ K<0.
\end{cases}
\end{equation*}

\end{cor}

\begin{proof} Since the boundary is convex, thus $\II\geq\sigma\geq0$. Thus

$$\varphi(s)=\int^T_s \e^{-K(u-s)-\sigma(l_u^x-l_s^x)} \d u\leq
\int^T_s \e^{-K(u-s)} \d u=\frac{1}{K} \left[1-\e^{-K(T-s)}\right]$$
and
$$\aligned A(s)&=\int^s_0\varphi(u) \e^{-K(s-u)-\sigma(l_s^x-l_u^x)}\d u\leq \frac{1}{K}\int^s_0 \left[1-\e^{-K(T-u)}\right]\e^{-K(s-u)}\d u\\
&=\frac{1}{2K^2}\bigg[2-2e^{-Ks}+e^{-K(T+s)}-e^{-K(T-s)}\bigg].\endaligned$$
Then we get
$$A(0)=0,\quad A(T)=\frac{1}{2K^2}\big(1-e^{-K}\big)^2$$
In the following, similar to the argument of Proposition \ref{p2.1}.  Taking the derivative of $s\rightarrow A(s)$ gives
$$A'(s)=\frac{1}{2K}\bigg[2e^{-Ks}-e^{-K(T+s)}-e^{-K(T-s)}\bigg].$$
Thus,\begin{equation}\label{eq3.7}A'(0)=\frac{1}{K}\big[1-e^{-K}\big]\geq0, \quad A'(T)=-\frac{1}{2K}\big[e^{-K}-1\big]^2.\end{equation}
Noting that

\begin{equation}\label{eq3.8}
 \left\{ \begin{array}{ll}
A'(T)>0  \qquad   &if ~K<0,\\
A'(T)<0 & if ~K>0.
\end{array}\right.
\end{equation}

Now we look for $s\in [0,T]$ such that $A'(s)=0$. We have

\begin{equation}\label{c3.1}\aligned &A'(s)=0 \\
&\Leftrightarrow~~2e^{-Ks}-e^{-K(T+s)}-e^{-K(T-s)}=0\\
&\Leftrightarrow~~2-e^{-KT}-e^{-KT+2Ks}=0\\&\Leftrightarrow~~e^{2Ks}
=2\e^{KT}-1.\endaligned\end{equation}
Therefore there exists at most one $t$ such that  $A'(s)=0$.
For the case where $K<0$, if there exists $s_0\in(0,T)$ such that $A(s_0)<0$.
Then by \eqref{eq3.7}  and \eqref{eq3.8}, the equation $A'(s)=0$ has at least two solutions, it is impossible.
Therefore for $K<0$, $A'(s)\geq0$. For $K>0$, we suppose $s_0$ such that $\Lambda'(t_0,T)=0$, then by (\ref{c3.1})
$$e^{2Ks_0}
=2\e^{KT}-1.$$
The proof is completed. 

\end{proof}

\beg{cor} Let $M$ be a Ricci-flat Riemannian manifold with a convex boundary, then
we have
\begin{equation*}
\aligned
\mathbb{E}\bigg(F^2\log \frac{F^2}{\|F\|^2_{L^2}}\bigg)\leq&T^2\int_{W_x^T(M)}\int^T_0|\nabla F|^2\d t\d \P_x.\endaligned\end{equation*}
\end{cor}

\beg{thebibliography}{99}

\leftskip=-2mm
\parskip=-1mm


\bibitem{AE95} S. Aida and K. D. Elworthy, \emph{Differential calculus on path and loop spaces. I. Logarithmic
Sobolev inequalities on path spaces,} C. R. Acad. Sci. Paris S\'erie I, 321(1995), 97--102.

\bibitem{BLW07}  D. Bakry, M. Ledoux, F.-Y. Wang, \emph{Perturbations of functional inequalities using growth conditions,}
J. Math. Pures Appl. 87(2007), 394--407.



\bibitem{CHL97} B. Capitaine,  E. P. Hsu and M. Ledoux, \emph{Martingale representation and a simple proof of logarithmic
Sobolev inequalities on path spaces,} Electron. Comm. Probab. 2(1997),
71--81.



\bibitem{CW14} X. Chen, B. Wu, \emph{Functional inequality on path space
over a non-compact Riemannian manifold,}  J.
Funct. Anal. 266(2014), 6753-6779.

\bibitem{CT18} L. J. Cheng, A. Thalmaier, \emph{Spectral gap on Riemannian path space over static and evolving manifolds,} J. Funct. Anal. 274(2018) 959¡V984

\bibitem{CM96} A. B. Cruzeiro and P. Malliavin, \emph{Renormalized differential geometry
and path space: structural equation, curvature,}  J. Funct. Anal. 139: 1 (1996), 119--181.

\bibitem{D92} B. K. Driver, \emph{A Cameron-Martin type quasi-invariance theorem for Brownian
motion on a compact Riemannian manifolds,} J. Funct. Anal.
110(1992), 273--376.

\bibitem{DR92} B. K. Driver and M. R\"{o}ckner, \emph{Construction of diffusions on path and
loop spaces of compact Riemannian manifolds,} C. R. Acad. Sci. Paris
S\'eries I 315(1992), 603--608.

\bibitem{EL99} K. D. Elworthy, X.- M. Li and Y. Lejan, \emph{On The geometry of diffusion
operators and stochastic flows,} Lecture Notes in Mathematics, 1720(1999), Springer-Verlag.

\bibitem{EM97} K. D. Elworthy and Z.-M. Ma, \emph{Vector fields on mapping spaces and
related Dirichlet forms and diffusions,} Osaka. J. Math. 34(1997), 629--651.

\bibitem{ES95} O. Enchev and D. W. Stroock, \emph{Towards a Riemannian geometry on the path space over a Riemannian manifold,}
 J. Funct. Anal. 134 : 2 (1995), 392--416.

\bibitem{F94} S.- Z. Fang, \emph{Un in\'equalit\'e du type Poincar\'esur un espace de chemins,}
C. R. Acad. Sci. Paris S\'erie I 318(1994), 257--260.

\bibitem{FM93} S.- Z. Fang and P. Malliavin, \emph{Stochastic analysis on the path
space of a Riemannian manifold: I. Markovian stochastic calculus,} J. Funct. Anal. 118 : 1 (1993), 249--274.

\bibitem{FWW08} S.- Z. Fang, F.-Y. Wang and B. Wu, \emph{Transportation-cost
inequality on path spaces with uniform distance,} Stochastic. Process. Appl. 118: 12 (2008), 2181--2197.

\bibitem{FW17} S. Z. Fang,  B. Wu, \emph{Remarks on spectral gaps on
the Riemannian path  space,}  Electron. Commun. Probab. 22 (2017), no. 19, 1¡V13.

\bibitem{FOT10} M. Fukushima,  Y. Oshima and M. Takeda,
\emph{Dirichlet Forms and Symmetric Markov Processes,} Walter de
Gruyter, 2010.

\bibitem{H16}M. Hairer. \emph{The motion of a random string,} arXiv:1605.02192, pages 1-20, 2016.


\bibitem{H97} E. P. Hsu, \emph{Logarithmic Sobolev inequalites on path spaces over compact Riemannian manifolds,}
Commun. Math. Phys. 189(1997), 9-16.

\bibitem{H99} E. P. Hsu, \emph{Analysis on path and loop spaces,} in "Probability Theory and Applications"
(E. P. Hsu and S. R. S. Varadhan, Eds.), LAS/PARK CITY Mathematics Series, 6(1999), 279--347,
Amer. Math. Soc. Providence.


\bibitem{H02} E. P. Hsu, \emph{Multiplicative functional for the heat equation on manifolds with boundary,} Mich. Math. J. 50(2002),351--367.








\bibitem{N} A. Naber, Characterizations of bounded Ricci curvature on smooth and nonsmooth spaces, {\it arXiv: 1306.6512v4}.

\bibitem{RWZZ17} M. R\"{o}ckner, B. Wu, R.C. Zhu and R.X. Zhu,\emph{Stochastic Heat Equations with Values in a
Manifold via Dirichlet Forms}, submitted.

\bibitem{W04} F.- Y. Wang, \emph{Weak poincar\'{e} Inequalities on path
spaces,} Int. Math. Res. Not. 2004(2004), 90--108.

\bibitem{W09} F.-Y. Wang, \emph{Second fundamental form and gradient of Neumann semigroups,}  J. Funct. Anal. 256(2009), 3461--3469.

\bibitem{W11}  F.-Y. Wang, \emph{Analysis on path spaces over Riemannian manifolds with boundary,} Comm. Math. Sci. 9(2011),1203--1212.

\bibitem{W14} F.- Y. Wang, \emph{Analysis for diffusion processes on Riemannian manifolds,} World Scientific, 2014.

\bibitem{WW18} F.- Y. Wang and B. Wu, \emph{Pointwise Characterizations of Curvature and Second Fundamental Form on
Riemannian Manifolds,} accepted by SC.

\bibitem{W16} B. Wu, \emph{Characterizations of the upper bound of Bakry-Emery curvature,}  {\it arXiv: 1612.03714v1}

\end{thebibliography}
\end{document}